\documentclass[11pt, a4paper]{amsart}
\usepackage{amssymb}
\usepackage{amsmath}
\usepackage{amssymb}
\usepackage{amsthm}
\usepackage{bbm}
\usepackage{bm}
\usepackage{mathtools}
\usepackage{mathrsfs}
\usepackage[T1]{fontenc}
\usepackage[usenames,dvipsnames,svgnames,table]{xcolor}
\usepackage{esint}
\allowdisplaybreaks
\usepackage{combelow}
\usepackage{enumitem}
\setenumerate{label={\rm (\alph{*})}}
\usepackage[colorlinks=true,citecolor=black,urlcolor=black,
linkcolor=black]{hyperref}

\usepackage{graphicx}
\usepackage{tikz}

\usepackage{geometry}
\makeatletter
\newcommand*\bigcdot{\mathpalette\bigcdot@{.5}}
\newcommand*\bigcdot@[2]{\mathbin{\vcenter{\hbox{\scalebox{#2}{$\m@th#1\bullet$}}}}}
\makeatother

\newtheorem{theorem}{Theorem}%[section]
\newtheorem{lemma}[theorem]{Lemma}

\theoremstyle{remark}
\newtheorem{remark}[theorem]{Remark}
\newtheorem{example}[theorem]{Example}
\usepackage{booktabs}

\normalsize

\def\XXint#1#2#3{{\setbox0=\hbox{$#1{#2#3}{\int}$ }
		\vcenter{\hbox{$#2#3$ }}\kern-.6\wd0}}

\newcommand{\mres}{\mathbin{\vrule height 1.6ex depth 0pt width
		0.13ex\vrule height 0.13ex depth 0pt width 1.3ex}}

\newcommand{\bv}{\operatorname{BV}}

\newcommand{\dif}{\operatorname{d}\!}

\newcommand{\R}{\mathbb{R}}

\newcommand{\M}{\mathbb{R}^{N \times n}}
\newcommand{\Rn}{\mathbb{R}^{n}}
\newcommand{\RN}{\mathbb{R}^{N}}

\newcommand{\wstar}{\stackrel{*}{\rightharpoonup}}

\newcommand{\sobo}{\operatorname{W}}

\renewcommand{\geq}{\geqslant}
\newcommand{\lebe}{\operatorname{L}}
\newcommand{\hold}{\operatorname{C}}
\newcommand{\curl}{\operatorname{curl}}
\renewcommand{\leq}{\leqslant}

\usepackage{tikz-cd}
\newcommand{\cd}{\mathrm{CD}}

\newcommand{\restrict}{\begin{picture}(10,8)\put(2,0){\line(0,1){7}}\put(1.8,0){\line(1,0){7}}\end{picture}}

\AtEndDocument{\bigskip{\footnotesize%
		\noindent\textsc{Andrew Wiles Building, University of Oxford,\\ Radcliffe Observatory Quarter, Oxford, OX2 6GG, UK}\\
		\textit{E-mail address}, J.~Kristensen: \texttt{kristens@maths.ox.ac.uk}
		
		\vspace{2.4pt}
		\noindent\textsc{Centro di Ricerca Matematica Ennio de Giorgi, Scuola Normle Superiore,
			\\Piazza dei Cavalieri, 3, 56126 Pisa, IT, }   \\
		\noindent\textit{E-mail address}, B.~Rai\cb{t}\u{a}: \texttt{bogdanraita@gmail.com}
}}
\begin{document}
	\title[Concentration in $\bv$]{Concentration effects of  $\bv$ gradients\\ have gradient structure}
	\author[J. Kristensen]{Jan Kristensen}
	\author[B. Rai\cb{t}\u{a}]{Bogdan Rai\cb{t}\u{a}}
	\subjclass[2020]{Primary: 49J45, Secondary: 28B05}
	\keywords{Generalized Young measures, Quasiconvexity, Functions of bounded variation}
	%\date{\today}

	\begin{abstract}
		We prove that the concentration effects arising from weakly-* convergent sequences of gradients of
		maps of bounded variation have gradient structure. This is in stark contrast with the corresponding oscillation phenomena.
	\end{abstract}
	\maketitle
	Any probability measure valued map with finite $p$th moment can be identified with the oscillatory behaviour of a $p$-uniformly
	integrable sequence, in the sense that it is the (oscillation) Young measure generated by that sequence. Similar statements can be proved to
	describe the gap between weak and strong convergence in Lebesgue spaces in terms of families of probability measures.
	Much like convex functions can be used to characterize probability measures by Jensen's inequality, it was observed in \cite{KP,FMP}
	that quasiconvex functions characterize weak convergence of gradients of maps in reflexive Sobolev spaces. Although the analogous
	result for  gradients of $\bv$ maps holds true \cite{KrRi}, the phenomenon is quite different: while in the Sobolev case both
	oscillation and concentration measures are generated by gradients of Sobolev maps, it is known from \cite{Alberti0}
	(see \cite[Ex.~7.6]{Kristensen}) that any parametrized probability measure with finite first moment is the oscillation measure
	of a $\bv$ Young measure. In other words, the oscillation behaviour of an $\lebe^1$ bounded sequence of gradients can be as far
	from having gradient structure as a general $\lebe^1$ function can fail to be a gradient.
	This leads to the expectation that concentration effects in $\bv$ should also fail to have gradient structure in general.
	Surprisingly, this is not the case:
	\begin{theorem}\label{thm}
		Let $\Omega$ be a bounded Lipschitz domain in $\Rn$. Let $\nu=(\nu_x,\lambda,\nu_x^\infty)$
		be a Young measure generated by a sequence $Du_j\wstar Du$ in $\mathscr{M}(\Omega,\M )$, where $u$, $u_j\in\bv(\Omega,\RN ) $.
		Suppose that $\lambda(\partial\Omega)=0$ and let $( \phi_{j})$ be a smooth approximation of the identity. Then there exists a
		sequence $( v_j )$ in $\hold^\infty_c(\Omega,\RN )$ such that $\bigl( D(\phi_j \ast u + v_j) \bigr)$ generates
		$\bigl( \delta_{\overline{\nu}_x},\lambda,\nu_x^\infty \bigr)$.
	\end{theorem}
	We use the notation and terminology from \cite{AB,KrRi,KrRa_notes}. This result also holds true in a general framework
	of differential constraints and will appear in the final version of \cite{KrRa}.
	
	We emphasize that for the Young measure $\nu = \bigl( \nu_x , \lambda , \nu_{x}^{\infty} \bigr)$ in Theorem \ref{thm}, the Young measure
	$$
	\left( \nu_{x},|\bar{\nu}_{x}^{\infty}|\lambda , \delta_{\frac{\bar{\nu}_{x}^{\infty}}{|\bar{\nu}_{x}^{\infty}|}} \right)
	$$
	obtained by switching off the unnecessary concentration in $\nu$ need not be generated by a sequence of gradients. See Example \ref{ex}
	after the proof of Theorem \ref{thm}.
	
	Our result implies that the coupling of oscillation and concentration effects in $\bv$-Young measures is a nonlinear phenomenon.
	This is of course in contrast to the case of $\sobo^{1,p}$-Young measures, where one can use linear decomposition lemmas \cite{FMP,Kristensen}.
	
	The main new concept we introduce to prove our assertion is the \textbf{convex deficiency integrand}, defined for Lipschitz
	integrands $F\colon \M \to \R$ by
	$$
	\cd_F(z)\coloneqq\sup \biggl\{ F'(\xi)\bullet z \colon F\text{ is differentiable at }\xi\in\M \biggr\} ,
	$$
	where the bullet denotes the dot product in $\M$.
	Thus $\cd_F$ is the support function for the essential range of the derivative $F^\prime$. It is easy to check that we also have 
	\begin{equation}\label{equivdef}
		\cd_F(z)=\sup\left\{\dfrac{F(\xi+t z )-F(\xi)}{t}\colon \xi\in\R^{N\times n},\,t>0\right\}
	\end{equation}
	and hereby that $\cd_F$ is the smallest positively $1$-homogeneous and convex integrand such that
	\begin{equation}\label{subadd}
		F(\xi+z)\leq F(\xi)+\cd_F(z)
	\end{equation}
	holds for all $z,\,\xi\in\M$.
	
	The other concepts we use are well-known and include the \emph{recession integrand}, defined for continuous integrands $F\colon \M \to \R$ by 
	$$
	F^\infty(z) \equiv \lim_{t\rightarrow\infty}\frac{F(tz)}{t},
	$$
	whenever the limit exists in $\R$ uniformly in $\{|z|=1\}$, and of \emph{quasiconvexity} in the sense of Morrey \cite{Morrey}: A continuous
	integrand $F\colon \M \to \R$ is said to be quasiconvex if 
	$$
	F(z)\leq \int_{(0,1)^n} \! F(z+D\varphi(x))\dif x
	\quad\text{for }z\in\M ,\,\varphi\in\hold^\infty_c((0,1)^n,\RN ).
	$$
	We record the following auxiliary lemma that also explains our choice of terminology for convex deficiency integrand.
	\begin{lemma}\label{auxiliary}
		Let $F \colon \M \to \R$ be a quasiconvex integrand that admits a recession integrand $F^\infty$. Then $F^{\infty}(z) \leq \cd_F (z)$ for
		all $z \in \M$ with equaity when $z$ has rank one. Furthermore, $\cd_{F} = F^\infty$ when $F$ is convex.
	\end{lemma}
	
	\begin{proof}
		First we recall that quasiconvexity implies rank-one convexity \cite{Morrey} and that rank-one convex integrands of at most linear growth
		are globally Lipschitz continuous \cite{MorreyB}. In view of \eqref{equivdef},
		$$
		F^{\infty}(z) = \lim_{t \to \infty}\frac{F(tz)-F(0)}{t} \leq \cd_{F}(z)
		$$
		holds for all $z \in \M$. Next fix $z$ of rank one and note that for each fixed $\xi \in \M$ the univariate function $t \mapsto F(tz+ \xi)$
		is convex, hence for $s>0$,
		$$
		F^{\infty}(z) = \sup_{t>0} \frac{F(tz + \xi)-F(\xi )}{t} \geq \frac{F(sz+\xi )-F( \xi )}{s}
		$$
		and so taking supremum over $s>0$, $\xi \in \M$, yields by \eqref{equivdef} that $F^{\infty}(z) \geq \cd_{F}(z)$. The last argument
		obviously also yields the last equality, $\cd_F = F^\infty$, when $F$ is convex.
	\end{proof}
	
	We write $\lambda=\lambda^a\mathscr{L}^n\mres\Omega+\lambda^s$ for the Radon--Nikodym decomposition of $\lambda$. 
	
	\begin{proof}[Proof of Theorem \ref{thm}]
		It was proved in \cite{KrRi} that for some $\mathscr{L}^n$ negligible set $N^a \subset \Omega$ we have
		$$
		F(\bar\nu_x+\lambda^a(x)\bar\nu_x^\infty)\leq \int_{\M} \! F\dif\nu_x+\lambda^a(x)\int_{\{|z|=1\}} \! F^\infty \dif\nu_x^\infty
		\quad\text{for } x\in\Omega \setminus N^a 
		$$
		for all quasiconvex integrands $F$ that admit a recession integrand $F^\infty$.
		Suppose we are at a point of \emph{diffuse concentration}, i.e., where $t \equiv \lambda^a(x)>0$. Consider for $\varepsilon > 0$ and
		$\xi \in \M$ the integrand
		$$
		f_{\varepsilon}(z)=\dfrac{F(\xi+\varepsilon z )-F(\xi)}{\varepsilon},
		$$
		which is itself quasiconvex and $f_\varepsilon^\infty=F^\infty$. If $\xi$ is a point of differentiability of $F$, then 
		$$
		\lim_{\varepsilon\searrow 0} f_\varepsilon(z)= F'(\xi) \bullet z \quad \mbox{ locally uniformly in } z.
		$$
		Plugging  $f_\varepsilon$ in the inequality above and passing to the limit $\varepsilon \searrow 0$, we obtain
		\begin{align*}
			F'(\xi) \bullet \bigl( \bar\nu_x+t\bar\nu_x^\infty\bigr) &\leq \int_{\M}\! F'(\xi)\bullet z \dif\nu_x(z)+t\int_{\{|z|=1\}}F^\infty \dif\nu_x^\infty\\
			&= F'(\xi)\bullet \bar\nu_x +t\int_{\{|z|=1\}} \! F^\infty \dif\nu_x^\infty,
		\end{align*}
		which is rearranged as $F'(\xi) \bullet \bar\nu_x^\infty \leq \langle F^\infty,\nu_x^\infty\rangle_{\mathcal M,\hold}$.
		Taking the supremum in $\xi$ on the left hand side, we arrive at the ``Jensen'' inequality
		\begin{align}\label{inequ}
			\cd_F(\bar\nu_x^\infty)\leq \int_{\{|z|=1\}} \! F^\infty \dif\nu_x^\infty\quad\text{for } x\in\Omega \setminus N^a \, \mbox{ with } \lambda^{a}(x) > 0.
		\end{align}
		For $x \in \Omega \setminus N^a$ with $\lambda^{a}(x) > 0$ we take $z=\bar{\nu}_{x}$, $\xi = \lambda^{a}(x)\bar{\nu}_{x}^{\infty}$ in \eqref{equivdef}
		to get for $F$ as above,
		$$
		F( \bar{\nu}_{x}+\lambda^{a}(x)\bar{\nu}_{x}^{\infty} ) \leq F( \bar{\nu}_{x}) + \cd_{F}( \bar{\nu}_{x}^{\infty})\lambda^{a}(x),
		$$
		and in combination with \eqref{inequ} we get
		\begin{align}\label{claim}
			F\bigl( \bar{\nu}_{x}+\lambda^{a}(x)\bar{\nu}_{x}^{\infty} \bigr) \leq F( \bar{\nu}_{x})+ \lambda^{a}(x) \int_{\{|z|=1\}} \! F^\infty \dif\nu_{x}^{\infty}.
		\end{align}
		Clearly this inequality extends to all $x \in \Omega \setminus N^a$ by declaring that $0$ times undefined is $0$.
		Hence we have obtained Jensen inequalities for all quasiconvex $F$ for which the recession integrand $F^\infty$ exists
		against the Young measure $\bigl( \delta_{\bar{\nu}_{x}},\lambda^a\mathscr L^n, \nu_x^\infty \bigr)$.
		We can then employ the construction in  \cite[Lem.~3.12]{KrRa} with $\mathscr A=\curl$ and $( \phi_j )$ a smooth mollifier
		to find maps $a_j \in \hold^{\infty}_{c}( \Omega , \RN )$ such that $a_j \wstar 0$ in $\bv ( \Omega , \RN )$ and
		$$
		\phi_j \ast (\bar{\nu}_{\cdot}+\lambda^a\bar{\nu}_{\cdot}^{\infty})+Da_j\text{ generates } \bigl( \delta_{\bar{\nu}_x},\lambda^a\mathscr L^n, \nu_{x}^{\infty} \bigr)
		$$
		On the other hand, we can use \cite[Lem.~3.11]{KrRa} with $\mathscr A=\curl$ to show that there also exists a sequence
		$b_j \in\hold^{\infty}_{c}(\Omega,\RN )$ with $b_j \wstar 0$ in $\bv ( \Omega , \RN )$ such that
		$$
		\phi_j \ast (\lambda^s\bar{\nu}_{\cdot}^{\infty})+Db_j\text{ generates } \bigl( \delta_0,\lambda^s, \nu_{x}^{\infty} \bigr) .
		$$
		Finally, \cite[Lem.~3.13]{KrRa} enables us to add the two generating sequences above and obtain the result by setting $v_j=a_j+b_j$.
	\end{proof}
	\begin{remark}
		If we combine Lemma \ref{auxiliary} with \eqref{inequ} we obtain an actual Jensen type inequality for the concentration-angle measure $\nu_{x}^\infty$
		at diffuse concentration points. Coupling this with Alberti's rank one theorem \cite{Alberti} and the main result of \cite{KiKr},
		we have that, for quasiconvex and positively $1$-homogeneous integrands $H$,
		$$
		H(\bar\nu_x^\infty)\leq \int_{\{|z|=1\}} \! H \dif\nu_x^\infty\quad\text{for }\lambda\text{-a.e. }x\in\Omega.
		$$
		This is another surprising consequence of our observation.
	\end{remark}
	
	\begin{example}\label{ex}
		Let $\mathrm{I}$ be the identity matrix in $\R^{2 \times 2}$ and put $\nu_{x} = \frac{1}{2}\bigl( \delta_{-\mathrm{I}}+\delta_{\mathrm{I}}\bigr)$
		for $x \in \Omega$, where $\Omega = B_{1}(0)$, the open unit disk in $\R^2$. Using \cite{Alberti0} we find a sequence $( u_j )$ of
		maps in $\hold^{\infty}_{c}( \Omega , \R^2 )$ such that $u_j \wstar 0$ in $\bv ( \Omega , \R^2 )$ and $(Du_{j})$ generates a Young
		measure $\nu = \bigl( \nu_x , \lambda , \nu_{x}^\infty \bigr)$. Let $F \colon \R^{2 \times 2} \to \R$ be a quasiconvex integrand
		that admits a recession integrand, a necessary assumption by \cite{Muller}. Then we have by general properties of Young measures that
		$\bar{\nu}_{x}\mathscr{L}^{2}\restrict \Omega +\bar{\nu}_{x}^{\infty}\lambda = 0$, hence $\bar{\nu}_{x}^{\infty} = 0$ for $\lambda$-a.e. $x$,
		and
		$$
		F(0) \leq \frac{F(\mathrm{I})+F(-\mathrm{I})}{2} + \lambda^{a}(x)\int_{\{ |z|=1 \}} \! F^{\infty} \, \dif \nu_{x}^{\infty} \quad
		\mbox{ for } \mathscr{L}^2 \mbox{-a.e. } x
		$$
		The Young measure obtained by switching off the unnecessary concentration in $\nu$ is $\nu^{o} \equiv \bigl( \nu_x , 0, \mbox{n/a} \bigr)$
		and since it is well-known that Jensen's inequality fails for the probability measure $\nu_x$ and some quasiconvex integrands $F$
		as above the Young measure $\nu^o$ cannot be generated by gradients.
	\end{example}


\begin{thebibliography}{99}
		
		\bibitem{Alberti0} Alberti, G., 1991. A Lusin type theorem for gradients. Journal of Functional Analysis, 100(1), pp.110-118.
		
		\bibitem{Alberti} Alberti, G., 1993. Rank one property for derivatives of functions with bounded variation.
		Proceedings of the Royal Society of Edinburgh Section A: Mathematics, 123(2), pp.239-274.
		
		\bibitem{AB} Alibert, J.J. and Bouchitt\'e, G., 1997. Non-uniform integrability and generalized young measure.
		Journal of Convex Analysis, 4, pp.129-148.
		
		\bibitem{FMP} Fonseca, I., M\"{u}ller, S. and Pedregal, P., 1998. Analysis of concentration and oscillation effects generated by gradients.
		SIAM journal on mathematical analysis, 29(3), pp.736-756.
		
		\bibitem{KP} Kinderlehrer, D. and Pedregal, P., 1994. Gradient Young measures generated by sequences in Sobolev spaces.
		The Journal of Geometric Analysis, 4(1), pp.59-90.
		
		\bibitem{KiKr} Kirchheim, B. and Kristensen, J., 2016. On rank one convex functions that are homogeneous of degree one.
		Archive for rational mechanics and analysis, 221(1), pp.527-558.
		
		\bibitem{Kristensen} Kristensen, J., 1999. Lower semicontinuity in spaces of weakly differentiable functions.
		Mathematische Annalen, 313(4), pp.653-710.
		
		\bibitem{KrRa} Kristensen, J. and Rai\cb{t}\u{a}, B., 2019. Oscillation and concentration in sequences of PDE constrained measures.
		arXiv preprint arXiv:1912.09190v1.
		
		\bibitem{KrRa_notes} Kristensen, J. and Rai\cb{t}\u{a}, B., 2020.
		\href{https://www.mis.mpg.de/publications/other-series/ln/4520.html}{\textit{An introduction to generalized Young measures}}.
		Lecture Notes 45/2020, Max Planck Institute for Mathematics in the Sciences.
		
		\bibitem{KrRi} Kristensen, J. and Rindler, F., 2010.
		Characterization of generalized gradient Young measures generated by sequences in $W^{1, 1}$ and $BV$.
		Archive for rational mechanics and analysis, 197(2), pp.539-598.
		
		\bibitem{Morrey} Morrey Jr., C.B., 1952. Quasi-convexity and the lower semicontinuity of multiple integrals.
		Pacific Journal of Mathematics, 2, pp.25-53.
		
		\bibitem{MorreyB} Morrey Jr., C.B., 1966. \textit{Multiple integrals in the calculus of variations.
			Reprint of the 1966 edition [MR0202511]}. Classics in Mathematics. Springer-Verlag, Berlin, 2008. x+506 pp.
		
		
		\bibitem{Muller} M\"{u}ller, S., 1992. On quasiconvex functions which are homogeneous of degree 1.
		Indiana University Mathematics Journal 41(1), pp.295-301.
		
		
	\end{thebibliography}
\end{document}